\documentclass[a4paper,11pt]{amsart}
\pdfoutput=1 
\usepackage{amsmath,amsthm,amssymb}
\usepackage[mathscr]{eucal}
 \usepackage{cite}
\usepackage{upgreek}
\usepackage[bookmarks=false]{hyperref}
\usepackage{enumerate}
\usepackage{amsmath}
\usepackage{mathrsfs}
\usepackage{blindtext}
\usepackage{scrextend}
\usepackage{enumitem}
\usepackage{bm} 
\addtokomafont{labelinglabel}{\sffamily}
\usepackage{color}
\usepackage[at]{easylist}
\usepackage{tikz}
\usepackage{svg}
\usepackage{graphicx}

\usepackage{comment}

\numberwithin{equation}{section}

\newtheorem{main}{Theorem}

\newtheorem{thm}{Theorem}[section]
\newtheorem*{thm*}{Theorem}
\newtheorem{lem}[thm]{Lemma}
\newtheorem*{prob*}{Problem}

\newtheorem{prop}[thm]{Proposition}
\newtheorem*{prop*}{Proposition}

\newtheorem{cor}[thm]{Corollary}
\newtheorem*{cor*}{Corollary}

\theoremstyle{definition}
\newtheorem{defn}[thm]{Definition}
\newtheorem*{defn*}{Definition}

\newtheorem{remark}[thm]{Remark}

\newtheorem*{question*}{Question}
\newtheorem*{Pquestion*}{Popa's question}

\newtheorem*{conv*}{Convention}

\newcommand{\M}{\mathbb{M}}
\newcommand{\N}{\mathbb{N}}
\newcommand{\R}{\mathbb{R}}
\newcommand{\C}{\mathbb{C}}

\newcommand{\T}{\mathbb{T}}

\newcommand{\cG}{\mathcal{G}}
\newcommand{\cH}{\mathcal{H}}

\newcommand{\cM}{\mathcal{M}}

\newcommand{\cU}{\mathcal{U}}

\newcommand{\ee}{\varepsilon}

\newcommand{\Real}{\operatorname{Re}}

\newcommand{\twoone}{II$_1$ }

\begin{document}

\title[Internal sequential commutation and single generation]
{ Internal sequential commutation and single generation}

\author[David Gao]{David Gao}
\address{Department of Mathematical Sciences, UCSD, 9500 Gilman Dr, La Jolla, CA 92092, USA}
\email{weg002@ucsd.edu}
\urladdr{https://sites.google.com/ucsd.edu/david-gao/home}

\author[Srivatsav Kunnawalkam Elayavalli]{Srivatsav Kunnawalkam Elayavalli}
\address{Department of Mathematical Sciences, UCSD, 9500 Gilman Dr, La Jolla, CA 92092, USA}\email{srivatsav.kunnawalkam.elayavalli@vanderbilt.edu}
\urladdr{https://sites.google.com/view/srivatsavke/home}

\author[Gregory Patchell]{Gregory Patchell}
\address{Department of Mathematical Sciences, UCSD, 9500 Gilman Dr, La Jolla, CA 92092, USA}\email{gpatchel@ucsd.edu}
\urladdr{https://sites.google.com/view/gpatchel/home}

\author[Hui Tan]{Hui Tan}
\address{Department of Mathematical Sciences, UCSD, 9500 Gilman Dr, La Jolla, CA 92092, USA}
\email{hutan@ucsd.edu}
\urladdr{https://tanhui3.wordpress.com/}

\begin{abstract}
We extract a precise internal description  of the sequential commutation equivalence relation introduced in \cite{patchellelayavalli2023sequential} for tracial von Neumann algebras. As an application we prove that if a tracial von Neumann algebra $N$ is generated by unitaries $\{u_i\}_{i\in \mathbb{N}}$ such that $u_i\sim u_j$ (i.e, there exists a finite set of Haar unitaries $\{w_i\}_{i=1}^{n}$ in $N^\mathcal{U}$ such that $[u_i, w_1]= [w_k, w_{k+1}]=[w_n,u_j]=0$ for all $1\leq k< n$) then $N$ is singly generated. This generalizes and recovers several known single generation phenomena for \twoone factors in the literature with a unified proof.   
\end{abstract}
\maketitle


\section{Introduction}
Building on recent developments in the structure theory of ultrapowers of \twoone factors (\cite{exoticCIKE, houdayer2023asymptotic, patchellelayavalli2023sequential}), a new structural framework in \twoone factors involving the notion of sequential commutation was introduced in \cite{patchellelayavalli2023sequential}. In \cite{patchellelayavalli2023sequential}, several properties and features  of this equivalence relation were  studied on the space of Haar unitaries (unitaries with $\tau(u^n)=0$ for all $n\in \mathbb{N}$) in a \twoone factor, with main applications to elementary equivalence problems (see also \cite{BCI15, AGKE, goldbring2023uniformly}).   In this paper we continue the study of this framework and as an application we prove a general result concerning single generation, extending several prior works in this theme, in particular those considered in  \cite{popacartansingle, PopaGeThin, GeShenGen, shen2009type}.

Now we state and describe the main result of the paper. Recall if $u,v\in \mathcal{U}(M)$ for a II$_1$ factor $M$, we write $u\sim^\cU v$ if there exist $k>0$ and Haar unitaries $u=w_0,w_1,\ldots,w_k=v$ such that $w_i\in M^\cU$ for each $i$ and $[w_i,w_{i+1}]=0$ for $i=0,\ldots,k-1.$  
\begin{main}
        If $M$ is generated by a countable set of unitaries $\{u_i\}_{i\in \mathbb{N}}$ such that $u_i\sim^\cU u_j$ for all $i\neq j$, then $M$ is singly generated.

\end{main}

Before we describe the proof technique, we document some examples of II$_1$ factors which admit a generating set consisting of sequentially commuting unitaries (we call this class $\mathcal{S}\mathcal{C})$: 

\begin{enumerate}
    \item II$_1$ factors with property Gamma. 
    \item Non prime II$_1$ factors. 
    \item II$_1$ factors with Cartan subalgebras. 
    \item $L(PSL_n(\mathbb{Z}))$, for $n\geq 3$. 
    \item II$_1$ factors with unique sequential commutation orbit $\mathfrak{O}(N)=1$ see  \cite{patchellelayavalli2023sequential} for examples. 
    \item II$_1$ factors with property C', see \cite{PopaPropC} for definition. 
\end{enumerate}

The single generation of the above families is not new in the settings of items (1) to (4): (1), (2) is due to \cite{PopaGeThin}, (3) is due to \cite{popacartansingle}, (4) is due to \cite{GeShenGen}. On the other hand the results in items (5) and (6) are new. Note also item (6) above recovers all items (1) through (4), in the sense that all the examples in (1) through (4) satisfy property C' (see \cite{GalatanPopa}). Additionally following results of \cite{shen2009type}, one can prove natural stability properties for single generation arising from factors from class $\mathcal{S}\mathcal{C}$, such as taking joins over diffuse intersections (see Remark \ref{stability under amalgams}). Interestingly all the examples above satisfy $h(N)\leq 0$, in the sense of 1-bounded entropy \cite{Hayes2018}. 

We now describe a crucial technical tool that we use in the proof of the main theorem. Developing certain discretization and lifting arguments, we extract a purely internal formulation of the sequential commutation equivalence relation which is equivalent to the original definition considered in \cite{patchellelayavalli2023sequential}.  Write $u\sim_k^+v$ if there exists $k$ so that for all $\ee>0$ there exist Haar unitaries $u=w_0,w_1,\ldots, w_k=v$ and unitaries $x_1,\ldots,x_{k-1}$, all in $M,$ such that, for each $i$, $\|x_i-w_i\|<\ee$, $x_i$ is finite-dimensional, $[u,x_1]=[x_{k-1},v]=0,$ and $[x_i,x_{i+1}] = 0.$ The following is shown in Proposition \ref{equivalence of internal and external}. 

\begin{prop*}
 Let $M$ be a \twoone factor and let $u,v\in M$ be Haar unitaries. Then $u\sim^\cU v$ if and only if there exists $k\in \mathbb{N}$ and a sequence $(v_n)_n$ of Haar unitaries in $M$ converging in the SOT to $v$ such that $u\sim^+_k v_n$ for all $n.$ \end{prop*}
    
The above characterization of sequential commutation is \emph{not} to be confused with the following naive internal formulation:  say $u\sim v$ if there exists $k>0$ and Haar unitaries $u=w_0,w_1,\ldots,w_k=v$ such that $w_i\in M$ for each $i$ and $[w_i,w_{i+1}]=0$ for $i=0,\ldots,k-1.$ Indeed, we show in Proposition \ref{thm-continuum-orbits} that in all separable II$_1$ factors, $\sim$ has continuum many orbits, while of course, $\sim^{\cU}$ has a unique orbit in several examples separable factors including those with property Gamma and the exotic factors from \cite{exoticCIKE}.

Now we describe some aspects of the proof of the main theorem. We are heavily inspired by the techniques of  \cite{GeShenGen}, \cite{shen2009type}. First, we move completely into the II$_1$ factor $N$ via our internal characterization of sequential commutation, but in addition, we can choose the commuting unitaries to be finite dimensional of large dimension, and small measure on all of their atoms. We then prove a refinement of an argument by
\cite{shen2009type} to properly leverage the sequential commutation of just a generating set (as opposed to all pairs of unitaries) while passing to the entire von Neumann algebra; and then essentially refine the strategy of \cite{GeShenGen} replacing the requirement of diffuseness of the sequential chain with a small atoms condition that we can afford; and finally conclude using patching and spectral calculus arguments. 



\subsection*{Acknowledgements:} The second author thanks David Sherman for a motivating conversation in Charlottesville in March 2024. We thank Adrian Ioana for helpful comments on the first draft. We thank Ben Hayes, Adrian Ioana and David Jekel for helpful conversations and encouragement.   

\section{Preliminaries}

Throughout, $(M,\tau)$ and $(M_n,\tau_n)$ will denote tracial von Neumann algebras and $\cU$ will denote a countably incomplete, or free, ultrafilter on $\N.$ $\prod_\cU M_n$ will denote the tracial ultraproduct over $\cU$ of the $(M_n,\tau_n).$ If a sequence $(x_n)_n$ of elements in $M_n$ represents an element $x\in \prod_\cU M_n$, we say that $(x_n)_n$ is a \emph{lift} of $x.$ We call a unitary $u\in M$ \emph{Haar} if $\tau(u^n) = 0$ for all non-zero integers $n.$ We write $\cH(M)$ to denote the set of Haar unitaries in $M.$ If $W^*(u)$, the algebra generated by $u$ in $M,$ is finite-dimensional we say $u$ is finite-dimensional; equivalently, $u$ has finite spectrum.

We first state without proof a general lifting result for normal elements, expanding upon Lemma 2.2 of \cite{patchellelayavalli2023sequential}.

\begin{lem}[Lemma 2.2 of \cite{patchellelayavalli2023sequential}]\label{Haar lifting}
    If $u\in \prod_\cU (M_n,\tau_n)$ is a Haar unitary and each $M_n$ is diffuse, then there are Haar unitaries $u_n\in M_n$ such that $u=(u_n)_n$.
\end{lem}

Let $x\in M$ be a normal element. By the spectral theorem, there is a projection-valued measure $\pi^x$ on the spectrum $\sigma(x)$ of $x$ such that $f(x) = \int_{\sigma(x)}f(\lambda) d\pi^x(\lambda)$ for all bounded Borel functions $f$ on $\sigma(x)$. We can define a measure $\mu^x = \mu^x_{\hat{1}}$ by $\mu^x(E) = \langle\pi^x(E)\hat{1},\hat{1}\rangle$ so that $\tau(f(x)) = \int_{\sigma(x)}f(\lambda)d\mu^x(\lambda)$ for all bounded Borel functions $f$ on $\sigma(x).$ 

We note that $\mu^x$ is uniquely determined by the moments $\tau(x^m (x^*)^n)$ of $x.$ This is because knowing these moments determines the values of $\tau(p(x))$ for all *-polynomials $p,$ which by the Stone-Weierstrass theorem allows us to determine $\tau(f(x))$ for all continuous functions $f$ on $\sigma(x)$, since $\sigma(x)$ is compact and $\tau$ is continuous. By the Riesz-Markov-Kakutani representation theorem, the function $\psi:C(\sigma(x)) \to \C$ given by $\psi(f) = \tau(f(x))$ is given by $\tau(f(x)) = \int_{\sigma(x)} f(\lambda)d\mu(\lambda)$ for a unique Borel probability measure $\mu$ on $\sigma(x)$. It is clear that $\mu^x$ is equal to this unique $\mu.$ We call $\mu^x$ the \emph{spectral measure} of $x$. The following is related to Lemma 4.4 in \cite{BenPT}, we include a proof of reader's convenience. 

\begin{lem}\label{borel-map-of-lift}
    Let $x\in \prod_\cU M_n$ be a normal element. Suppose $x = (x_n)_n$ is a lift of $x$ such that each $x_n$ is normal and has the same moments (and thus the same spectral measure) as $x.$ Then for any bounded Borel function on $\sigma(x)$, $(f(x_n))_n$ is a lift of $f(x).$
\end{lem}

\begin{proof}
    The result is immediate if $f$ is a *-polynomial by the ultrapower construction. Otherwise, take a sequence of *-polynomials $p_m$ such that $\|p_m-f\|_2 < 1/m$, where $\|g\|_2^2 = \int_{\sigma(x)} |g|^2 d\mu^x = \int_{\sigma(x_n)}|g|^2 d\mu^{x_n}$ (the 2-norm with respect to $\mu^x = \mu^{x_n}$).

    Let $(y_n)_n$ be a lift of $f(x)$. We note that for all $m,$ $p_m(x_n)$ is a lift of $p_m(x).$ Then $\lim_{n\to\cU} \|y_n - p_m(x_n)\|_2 = \|f(x) - p_m(x)\|_2 = \|f-p_m\|_2 <1/m.$ But we also have that $\|p_m(x_n) - f(x_n)\|_2 = \|p_m-f\|_2 < 1/m$ so that $\lim_{n\to\cU} \|y_n - f(x_n)\|_2 < 2/m.$ Since we can take $m$ arbitrarily large, we have that $(f(x_n))_n$ is a lift of $f(x).$
\end{proof}

As a corollary of the previous two lemmas, we have:

\begin{lem}\label{lem-normal-lifting}
    Let $(M_n)_n$ be a sequence of diffuse tracial von Neumann algebras and let $x\in \prod_\cU M_n$ be normal. Then $x$ lifts to a sequence $x_n\in M_n$ such that $\tau(x^r(x^*)^s) = \tau(x_n^r(x_n^*)^s)$ for all integers $r,s,n.$
\end{lem}

\begin{proof}
    $\cM = \prod_\cU M_n$ is diffuse, so there is a Haar unitary $u\in \cM$ so that $x\in W^*(u).$ Furthermore, by the Borel functional calculus there is a bounded Borel function $f$ such that $f(u) = x.$ Use Lemma \ref{Haar lifting} to lift $u$ to a sequence $(u_n)_n$ of Haars in $M_n.$ By Lemma \ref{borel-map-of-lift} $(f(u_n))_n$ is a lift of $x$ with all the same moments. 
\end{proof}

\section{Proofs of main results}

\subsection{Internal sequential commutation}

In a \twoone factor, every pair of finite dimensional subalgebras are unitarily conjugate. Therefore every pair of unitaries with the same finite-supported spectral measures are conjugate. If the pair of unitaries are moreover close in 2-norm, the conjugating unitary can be chosen close to identity. We first recall a technical lemma due to Connes \cite{connes1975outer} (see also Lemma 1.4 of \cite{Connes} or Lemma 2.2 of \cite{christensen1979subalgebrasoffinite}). 

\begin{lem}[Lemma 1.1.4 of \cite{connes1975outer}]\label{connes-estimate}
    Let $N$ be a countably decomposeable von Neumann algebra on a Hilbert space $\cH$ and let $\xi\in \cH$. If $p,q \in N$ are equivalent projections then there is a partial isometry $v\in N$ such that $vv^* = p,$ $v^*v = q,$ and $\|(v-p)\xi\| \leq 6\|(p-q)\xi\|$. If $N$ is a \twoone factor and $p,q\in N$ are projections such that $\tau(p)=\tau(q)$ then we can pick $v\in N$ a partial isometry such that $vv^*=p,$ $v^*v=q,$ and $\|v-p\|_2 \leq 3\|p-q\|_2$.
\end{lem}


\begin{lem}\label{conj-finite-unitaries}
    Fix $t>0$ and $0<\kappa < \pi/2.$ Then for all $\ee>0$ there is $\delta>0$ such that whenever $u_1,u_2 \in M$ are unitaries in a \twoone with the same purely atomic spectral measure $\mu$ where each atom of $\mu$ has measure at least $t$ and all distinct eigenvalues of $u_i$ have radian distance between them at least $\kappa$, and $\|u_1-u_2\|_2 < \delta$ then there is a unitary $v\in M$ such that $\|v-1\|_2 < \ee$ and $v^*u_1v = u_2.$
\end{lem}

\begin{proof}
    Let $\ee>0$ and set choose $\delta > 0$ such that $\delta < \frac{\ee^2}{6}\sqrt{t(1-\cos(\kappa)}$.

    Let $u_1,u_2\in M$ be unitaries as in the statement of the lemma. Write $u_i = \sum_{j=1}^n\lambda_jp_{ij}$. We compute that 
    \begin{align*}
        \|u_1-u_2\|_2^2 &= 2-2\Real\tau(u_1^*u_2) \\
        &= 2-2\Real \sum_{j,k=1}^n \overline{\lambda_j}\lambda_k \tau(p_{ij}p_{2k})\\
        &= 2-2\sum_{j=1}^n\tau(p_{ij}p_{2j}) - 2\Real \sum_{j\neq k}\overline{\lambda_j}\lambda_k \tau(p_{ij}p_{2k})
    \end{align*}

    Note that by hypothesis, for $j\neq k$ $\mathrm{Re}(\overline{\lambda_j}\lambda_k) \leq \cos(\kappa).$ Hence
    \begin{align*}
        \|u_1-u_2\|_2^2&\geq 2 - 2\sum_{j=1}^n\tau(p_{ij}p_{2j}) - 2\Real \sum_{j\neq k}\cos(\kappa) \tau(p_{ij}p_{2k})\\
        &= 2 - 2\cos(\kappa) - 2(1-\cos(\kappa))\sum_{j=1}^n\tau(p_{ij}p_{2j})
    \end{align*}

    By hypothesis $\|u_1-u_2\|_2 < \delta,$ so that 
    $$\sum_{j=1}^n\tau(p_{ij}p_{2j}) \geq 1 - \frac{\delta^2}{2-2\cos(\kappa)}.$$
    This in turn implies that, since $\tau(p_{1j}) = \tau(p_{2j})$ for all $j,$
    $$\sum_{j=1}^n \|p_{1j}-p_{2j}\|_2^2 \leq \frac{\delta^2}{1-\cos(\kappa)}. $$

    By Lemma \ref{connes-estimate} there are partial isometries $v_j$ such that $v_jv_j^* = p_{1j}$, $v_j^*v_j = p_{2j}$, and $\|v_j-p_{1j}\|_2 \leq 6\|p_{1j}-p_{2j}\|_2$. Set $v = \sum_{j=1}^n v_j$. Then $v$ is a unitary such that $v^*u_1v = u_2$.

    Then we can write \begin{align*}
        \|v-1\|_2^2 &= 2-2\Real\tau(v) \\
        &= 2-2\Real \sum_{j=1}^n\tau(v_j)\\
        &= 2\Real \sum_{j=1}^n\tau(p_{1j}-v_j).
    \end{align*}
    Applying the Cauchy-Schwarz inequality and the fact that $n\leq 1/t$ (otherwise the atoms of $\mu$ would have total weight greater than 1), we get that 
    $$\|v-1\|_2^2 \leq 2\sqrt{n}\left(\sum_{j=1}^n\|p_{1j}-v_j\|_2^2\right)^{1/2} \leq 6\sqrt{n}\left(\sum_{j=1}^n\|p_{1j}-p_{2j}\|_2^2\right)^{1/2}\leq\frac{6\delta}{\sqrt{t(1-\cos(\kappa))}} < \ee^2,$$
    completing the proof.
\end{proof}

We now use our technical lemmas to characterize commutation of Haar unitaries in an ultrapower of $M$ by large-dimensional unitaries in $M.$

\begin{lem}\label{lift-to-commuting-almost-haars}
    Let $M$ be a \twoone factor. If $u,v \in M^\cU$ are commuting Haar unitaries then for any lift $u_n$ of $u$ to Haar unitaries and any $\ee>0$ there is a lift of $v$ to Haar unitaries $v_n$ and there are finite-dimensional unitaries $w_n \in M$ s.t. all atoms of the spectral measure of $w_n$ have measure less than $\ee$ and $[u_n,w_n]=[w_n,v_n] = 0.$ Moreover, for any $N$ with $\frac{1}{N} < \ee$, we may choose each $w_n$ so that $w_n \in W^*(u_n)$, $\|u_n-w_n\|<\ee\pi$ and all atoms of $w_n$ have measure exactly $\frac{1}{N}$.
\end{lem}

\begin{proof}
    Fix $\ee>0.$ Let $N$ be such that $1/N < \ee.$ Let $f:\T\to\T $ be an \emph{$N$-uniform discretization}, a Borel function with finite image such that the measure of the preimage of each point in the range is an arc in $\T$ with measure equal to $1/N.$ Set $w_n = f(u_n)$ and $w = f(u)$ so that $(w_n)_n$ is a lift of $w,$ $w_n$ commutes with $u_n$ (in fact $\|u_n-w_n\| < \ee\pi$ too if one chooses $f$ appropriately) and $w$ commutes with $v.$ 

    Since $w$ and $v$ commutes in $M^\cU,$ they generate a diffuse abelian subalgebra. Lemma \ref{lem-normal-lifting} guarantees commuting unitary lifts $x_n$ and $v_n$ of $w,v$ respectively such that $x_n$ has the same moments as $w$ and the $v_n$ are all Haar.

    Set $t = 1/N$ and $\kappa = 2\pi/N.$ For $\ee = 1/k$, $k$ an integer, choose $\delta_k>0$ as guaranteed by Lemma \ref{conj-finite-unitaries}. Without loss of generality we can choose the $\delta_k$ to be decreasing to 0. Since $(w_n)_n$ and $(x_n)_n$ are both lifts of $w$, the decreasing sequence of sets $A_k = \{n\in\N : \|w_n-x_n\|_2 < \delta_k\}$ are all in $\cU.$ Since $\cU$ is countably incomplete, we may choose a sequence of sets $B_k \in \cU$ decreasing to $\varnothing$, so $A'_k = A_k \cap B_k$ is a decreasing sequence of sets in $\cU$ with $\cap_k A'_k = \varnothing$. For $n \in A'_k \setminus A'_{k+1}$, apply Lemma \ref{conj-finite-unitaries} to get unitaries $y_n$ such that $\|y_n-1\|_2 < 1/k$ and $y_n^*x_ny_n = w_n.$ For $n \notin A_1$, let $y_n = 1$. Then as $\cap_k A'_k = \varnothing$, we have defined $y_n$ for all $n$. We note that $\|y_n-1\|_2 \to 0$ as $n\to\cU$ and so $(y_n^*v_ny_n)_n$ is a Haar lifting of $v.$ Then the sequences of unitaries $(w_n)_n$ and $(y_n^*v_ny_n)_n$ satisfy the conclusion of the lemma.
\end{proof}


\begin{defn}\label{seq-comm-def} The following are three notions of sequentially commutation in a tracial von Neumann algebra $(M,\tau)$. The first two conditions appear in \cite{patchellelayavalli2023sequential}. Throughout, let $u,v\in M$ be Haar unitaries.

\begin{enumerate}
    \item[(a)]  We write $u\sim_k v$ if there exist Haar unitaries $u=w_0,w_1,\ldots,w_k=v$ such that $w_i\in M$ for each $i$ and $[w_i,w_{i+1}]=0$ for $i=0,\ldots,k-1.$  We write $u\sim v$  if there exists $k$ such that $u\sim_k v$.
    \item[(b)]  We write $u\sim_k^\cU v$ if there exist Haar unitaries $u=w_0,w_1,\ldots,w_k=v$ such that $w_i\in M^\cU$ for each $i$ and $[w_i,w_{i+1}]=0$ for $i=0,\ldots,k-1.$  We write $u\sim^\cU v$ and say that $u,v\in M$ \emph{sequentially commute} if there exists $k$ such that $u\sim_k^\cU v$.
    \item[(c)]  We write $u\sim_k^+v$ if there exists $k$ so that for all $\ee>0$ there exist Haar unitaries $u=w_0,w_1,\ldots, w_k=v$ and unitaries $x_1,\ldots,x_{k-1}$, all in $M,$ such that, for each $i$, $\|x_i-w_i\|<\ee$, $x_i$ is finite-dimensional, each atom of each $x_i$ have the same measure, $[u,x_1]=[x_{k-1},v]=0,$ and $[x_i,x_{i+1}] = 0.$ We write $u\sim^+ v$ and say that $u,v\in M$ \emph{internally sequentially commute} if there exists $k$ such that $u\sim_k^+ v$.
\end{enumerate}
\end{defn}

We will now show that (b) and (c) are related, meaning we can characterize sequential commutation in the sense of \cite{patchellelayavalli2023sequential} via actually commuting unitaries without appealing to an ultrapower. We will then see that (a) is a distinct criterion.

\begin{prop}\label{equivalence of internal and external}
    Let $M$ be a \twoone factor and let $u,v\in M$ be Haar unitaries. Then $u\sim^\cU v$ if and only if there exists $k\in \mathbb{N}$ and a sequence $(v_n)_n$ of Haar unitaries in $M$ converging in the SOT to $v$ such that $u\sim^+_k v_n$ for all $n.$ 
\end{prop}

\begin{proof}
    Suppose there are Haars $v_n$ converging to $v$ in the SOT and $k$ so that $u\sim_k^+ v_n$ for all $n.$ Then for each $n\geq 1$ and $i=1,\ldots,k-1$ there are finite-dimensional unitaries $x_{i,n}$ such that $[u,x_{1,n}] = [x_{i,n},x_{i+1,n}]=[x_{k-1,n},v_n] = 0$ and $\|x_{i,n}-w_{i,n}\| < 1/n$ for some Haar unitary $w_{i,n}\in M$. Set $x_i = (x_{i,n})_n \in M^\cU.$ Then each $x_i$ is a Haar unitary, and since $v = (v_n)_n$ as elements of $M^\cU,$ we have that $[u,x_1] = [x_i,x_{i+1}] = [x_{k-1},v]$ as required. Hence $u\sim^\cU v.$

    Conversely, suppose $u\sim^\cU v$. Then there is $k$ and Haar unitaries $u=w_0,w_1,\ldots,w_k=v$ in $M^\cU$ as in Definition \ref{seq-comm-def}(b). Fix $\ee>0$ and apply Lemma \ref{lift-to-commuting-almost-haars} repeatedly to get Haar lifts $w_i = (w_{i,n})_n$ and unitaries $x_{i,n} \in W^*(w_{i,n})$ so that $\|w_{i,n}-x_{i,n}\|<\ee$ and $[x_{i,n},w_{i+1,n}] = 0$ for all $n$ and $i=0,\ldots,k-1.$ In particular, we note that $[x_{i,n},x_{i+1,n}]=0$ too. The lift of $u=w_0$ can be chosen to be the constant sequence $(u)_n$ and the lift of $v=w_k$ gives Haar unitaries $w_{k,n}$ which converge to $v$ in the SOT. 
\end{proof}

\begin{cor}
    Let $(M,\tau)$ be a \twoone factor. Then $u\sim^\cU v$ for all Haar unitaries $u,v\in M$ if and only if there is an SOT-dense subset $D\subset \cH(M)$ of the Haar unitaries of $M$ such that $u\sim^+v$ for all $u,v\in D.$ Furthermore, $D$ can be chosen to contain any particular Haar unitary $u_0 \in M.$
\end{cor}

In \cite{PopaWeakInter}, it is shown that separable \twoone factors always contain coarse MASAs; i.e., if $M$ is a separable \twoone factor then it contains a maximal abelian subalgebra $A$ such that the $A$-$A$ bimodule $L^2(M)\ominus L^2(A)$ is a multiple of the coarse bimodule $L^2(A)\otimes L^2(A).$ We can use coarseness to show that Definition \ref{seq-comm-def}(a) is distinct from Definition \ref{seq-comm-def}(b), namely, sequential commutation for separable \twoone factors. Note that many separable \twoone factors have only one $\sim^\cU$-orbit, such as any factor with Property Gamma. We remark that moreover statement of the Proposition below is already known and is documented in \cite{sorinsingular}. We thank Adrian Ioana for suggesting the idea for the following proof. 

\begin{prop}\label{thm-continuum-orbits}
    If $M$ is a separable \twoone factor, then $\sim$ has continuum many orbits. Moreover, there are continuum many Haar unitaries in $M$ which are pairwise non-conjugate.  
\end{prop} 

\begin{proof}
    Embed a diffuse abelian von Neumann algebra $A$ in $M$ such that $L^2(M) \ominus L^2(A)$ is a coarse $A$-$A$-bimodule. Let $u$ be a Haar unitary in $A$.  Let $e:L^2(M) \to L^2(M)\ominus L^2(A)$ denote the orthogonal projection onto $L^2(M)\ominus L^2(A).$ Let $v\in M$ be a unitary such that $v^*uv \in A.$ Observe that $(u^n)_n$ be a sequence of unitaries in $A_\lambda$ going weakly to 0. We now observe that by the $A$-$A$-bimodularity of the map $e,$ $u^ne(v) = e(u^nv)= e(vv^*u^nv) = e(v)v^*u^nv$. Therefore
    \begin{align*}
        \|e(v)\|_2^2 &= \langle u^{-n}u^ne(v), e(v) \rangle\\
        &= \langle u^{-n}e(v)v^*u^nv, e(v) \rangle.
    \end{align*}
    We also note that since $L^2(M)\ominus L^2(A)\cong \bigoplus L^2(A)\otimes L^2(A)$ is a coarse $A$-$A$-bimodule, it is densely spanned by direct sums of vectors of the form $\xi\otimes \eta$. But we see that
    \begin{align*}
        \langle u^n\xi\otimes \eta v^*u^nv,\xi\otimes \eta\rangle &= \langle u^n\xi,\xi\rangle\langle\eta v^*u^nv,\eta\rangle\to 0 
    \end{align*}
    since $u^n\to 0$ weakly. We thus see that $\langle u^{-n}e(v)v^*u^nv, e(v) \rangle\to 0 $ and therefore $\|e(v)\|_2^2 = 0.$ Hence $v \in A.$ 

    For the moreover, we observe that if $W^*(u) = A$ then $\lambda u$ and $\mu u$ are non-conjugate for all $\lambda\neq \mu \in \T.$ We also observe that if $u\in A$ is any Haar unitary and $w\in M$ is Haar such that $u\sim w$, then necessarily $w\in A.$
     
     Now let $v$ be a Haar unitary in $M\setminus A$. Write $v = e^{ih}$ for some self-adjoint operator $h\in M\setminus A.$ Define $v_t = e^{ith}$ for $t\in \R.$ Since $v\not\in A,$ by continuity $v_t \not\in A$ for all $t$ in some non-trivial interval $(1-1/N,1+1/N)$. Since $v_{s+t}=v_sv_t,$ we see that $v_t\not\in A$ for all $0\neq t\in (-1/N,1/N).$ Therefore for all $s,t \in (-1/(2N),1/(2N))$, $v_s^*v_t\in A$ only if $s=t.$ Therefore $v_s^*v_t u v_t^*v_s \not\in A$ for any such distinct $s,t,$ by the coarseness of $A$ in $M$. This implies that for $t\in (-1/(2N),1/(2N))$, the $\sim$-orbits of $v_t u v_t^*$ are all distinct. 
\end{proof}

\subsection{Single generation arguments}

The proof of the following lemma is inspired by \cite{shen2009type}.

\begin{lem}\label{exist-of-hyperfinite-generator}
    Let $M$ be a \twoone factor and $u \in M$ a Haar unitary. Then for all $1 \geq \ee > 0$ there is $R \subset M$ irreducible hyperfinite subfactor such that $\{R,u\}''$ is generated by a unitary $v$ contained in a matrix algebra $\M_N(\C) \subset R$ and a self-adjoint $T$ whose support is majorized by a projection $p \in \M(\C)$ with trace less than $\ee.$ 
\end{lem}

\begin{proof}
    Choose an integer $N > 0$ large enough so that $N \leq \frac{\lfloor \ee N/2 \rfloor(\lfloor \ee N/2 \rfloor - 1)}{2}$ and $\frac{1}{N} < \frac{\ee}{6}$. Choose $N$ orthogonal projections $\{p_i\}_{i=1}^N$ in $W^*(u)$ which sum to 1, whose traces are all $\frac{1}{N}$, and
    \begin{equation*}
        u = \sum_{i=1}^N p_iup_i.
    \end{equation*}

    By Corollary 4.1 of \cite{Pop81kadison}, there exists an irreducible subfactor $R \subset M$. By conjugating by a unitary, we may assume $p_i \in R$ for all $i$. Clearly, there is a matrix subalgebra $\M_N(\C) \subset R$ such that $p_i \in \M_N(\C)$ for all $i$. The $p_i$ may simply be regarded as the matrix units along the diagonal. We denote by $e_{ij}$ the other matrix units in $\M_N(\C).$ Then $R = \M_N(\C) \otimes p_1Rp_1$. The algebra $p_1Rp_1$ is another copy of the hyperfinite \twoone factor, which is generated (together with its unit) by two self-adjoints $S_1, S_2$. $\M_N(\C)$ is generated by $p_1$ and a single shift unitary $v = \sum_{i=1}^{N-1} e_{i(i+1)} + e_{N1}$, so $R = \M_N(\C) \otimes p_1Rp_1$ is generated by
    \begin{equation*}
        v, p_1, S_1, S_2.
    \end{equation*}

    By adding appropriate scalar multiples of $p_1$ to $S_1$ and $S_2$, we may assume without loss of generality that the spectra of $S_1$ and $S_2$ (within $p_1Rp_1$) are disjoint and are both disjoint from $\{1\}$. Then by the functional calculus, one easily sees that $R = \M_N(\C) \otimes p_1Rp_1$ is generated by
    \begin{equation*}
        v, T_1 = p_1 + e_{21}S_1e_{12} + e_{31}S_2e_{13}.
    \end{equation*}

    Note that $T_1$ is a self-adjoint with support majorized by $p_1 + p_2 + p_3$, which has trace $\frac{3}{N} < \frac{\ee}{2}$.

    Now, consider the region of $\M_N(\C)$ supported on $P = \sum_{i=4}^{3+\lfloor \ee N/2 \rfloor} p_i$. The upper triangular area of this region has $\frac{\lfloor \ee N/2 \rfloor(\lfloor \ee N/2 \rfloor - 1)}{2}$ blocks. Since we have chosen $N$ so that $N \leq \frac{\lfloor \ee N/2 \rfloor(\lfloor \ee N/2 \rfloor - 1)}{2}$, there exists an injection $\pi: \{1, \cdots, N\} \to \{(i,j) \in \{4, \cdots, 3+\lfloor \ee N/2 \rfloor\}^2: i < j\}$. Let $\pi(i) = (\pi_1(i), \pi_2(i))$. We may thus define a self-adjoint $T_2$ by
    \begin{equation*}
        T_2 = \sum_{i=1}^N (e_{\pi_1(i)i}ue_{i\pi_2(i)} + e_{\pi_2(i)i}u^\ast e_{i\pi_1(i)}).
    \end{equation*}

     Because $e_{\pi_1(i)i}ue_{i\pi_2(i)}$ all lie within the upper triangular area while $e_{\pi_2(i)i}u^\ast e_{i\pi_1(i)}$ all lie within the lower triangular area, and because $u = \sum_{i=1}^N p_iup_i$ only has components along the diagonal, we easily see that $u \in \{\M_N(\C), T_2\}''$. $T_2$ has support majorized by $P$. By adding appropriate scalar multiples of $P$ to $T_2$, we may assume $T_2$ has spectrum (within $P\{R, u\}''P$) disjoint from the spectrum of $T_1$. Since $T_1$ is supported within $p_1 + p_2 + p_3$, which is orthogonal to $P$, again by functional calculus, we see that $\{R, u\}''$ is generated by
     \begin{equation*}
         v, T = T_1 + T_2.
     \end{equation*}

     $T$ has support contained within $p = p_1 + p_2 + p_3 + P$. Since $\tau(p_1 + p_2 + p_3) < \frac{\ee}{2}$ and we easily see that $\tau(P) \leq \frac{\ee}{2}$, we have $\tau(p) < \ee$, as desired.
\end{proof}

The following lemma is a generalization of Proposition 1 of \cite{GeShenGen} and follows nearly the same proof.

\begin{lem}\label{generator-seq-commuting}
    For all $\ee > 0$ and $n > 0$ there is $N > 0$ such that if $u_1,\ldots,u_n$ are unitaries in a \twoone factor $M$ satisfying
    \begin{enumerate}
        \item $u_{i+1}^*u_iu_{i+1} \in W^*(u_1,\ldots,u_i)$ for all $i = 1, \cdots, n-1$; and
        \item For each $i$, $u_i$ is either diffuse; or is finite-dimensional, has all its atoms having the same trace, and has the number of atoms being a multiple of $N$;
    \end{enumerate}
    then for any irreducible hyperfinite subfactor $R \subset M$, $W^*(R,u_1,\ldots,u_n)$ is generated by $R,$ $u_1,$ and a self-adjoint whose support is majorized by a projection $p \in R$ with trace less than $\ee.$
\end{lem}

\begin{proof}
    By direct calculations, we see that the positive solutions to the equation $\frac{m(m-1)}{2} \geq k$ are given by $m \geq \frac{1+\sqrt{1+8k}}{2} = \frac{1}{2} + \sqrt{\frac{1}{4} + 2k}$. We also see that $\sqrt{\frac{1}{4} + 2k} - \sqrt{2k} \leq \frac{1}{2}$ for all $k \geq 0$. Note that $\sqrt{\frac{1}{4} + 2k} - \sqrt{2k} \leq \frac{1}{2}$ may be rearranged to $\frac{1+\sqrt{1+8k}}{2} \leq \sqrt{2k} + 1$. So, for any $k \geq 0$, there is an integer $m$ such that $\frac{m(m-1)}{2} \geq k$ and $m \leq \sqrt{2k} + 2$. Now choose $K > 0$ such that $\sum_{i=1}^{n-1} (\sqrt{\frac{2}{K^i}} + \frac{2}{K^i}) < \ee$ and set $N = K^{n-1}$.
    
    By the conditions on $u_i$, we may choose, for each $i = 1, \cdots, n-1$, pairwise orthogonal projections $F^{(i)}_1, \cdots, F^{(i)}_{K^i}$ in $\{u_i\}''$ all of trace $\frac{1}{K^i}$. Fix unital embeddings $\M_K(\C) \subset \M_{K^2}(\C) \subset \cdots \subset \M_{K^{n-1}}(\C) \subset R$ with the embeddings being the standard diagonal ones. Let $\{E^{(i)}_{\alpha\beta}\}_{1 \leq \alpha, \beta \leq K^i}$ be the standard system of matrix units for $\M_{K^i}(\C)$. Note that the embeddings between the matrix algebras are standard diagonal embeddings, so $E^{(i)}_{jj} = \sum_{k=(j-1)K+1}^{jK} E^{(i+1)}_{kk}$.
    
    Let $R_i = \{R, u_1, \cdots, u_i\}''$. As $R \subset M$ is irreducible, $R_i$ is a factor for all $i$. Then by the conditions of $F^{(i)}_j$ and $E^{(i)}_{jj}$, we see that there exist unitaries $W_i \in R_i$ so that $W_iF^{(i)}_jW_i^\ast = E^{(i)}_{jj}$ for all $j = 1, \cdots, K^i$. Because $u_{i+1}^\ast u_i u_{i+1} \in W^\ast(u_1, \cdots, u_i) \subset R_i$, and because $F^{(i)}_1, \cdots, F^{(i)}_{K^i} \in \{u_i\}''$, we have $u_{i+1}^\ast F^{(i)}_j u_{i+1} \in R_i$, so there exists unitary $V_i \in R_i$ s.t. $V_iE^{(i)}_{jj}V_i^\ast = u_{i+1}^\ast F^{(i)}_j u_{i+1} = u_{i+1}^\ast W_i^\ast E^{(i)}_{jj} W_i u_{i+1}$ for all $j$. So $W_i u_{i+1}V_i$ commutes with $E^{(i)}_{jj}$ for all $j$; i.e.,
    \begin{equation*}
        W_i u_{i+1}V_i = \sum_{j=1}^{K^i} E^{(i)}_{jj}W_i u_{i+1}V_iE^{(i)}_{jj}.
    \end{equation*}

    Choose an integer $m_i$ such that $\frac{m_i(m_i-1)}{2} \geq K^i$ and $m_i \leq \sqrt{2K^i} + 2$. Let $M_i = \sum_{j=1}^{i-1} m_jK^{i-j}$. Then, there exists an injection $\pi_i: \{1, \cdots, K^i\} \to \{(\alpha,\beta) \in \{M_i+1, \cdots, M_i+m_i\}^2: \alpha < \beta\}$. Let $\pi_i(j) = (\pi_{i1}(j), \pi_{i2}(j))$. We may thus define a self-adjoint $T_i$ by
    \begin{equation*}
        T_i = \sum_{j=1}^{K^i} E^{(i)}_{\pi_{i1}(j)j}W_i u_{i+1}V_iE^{(i)}_{j\pi_{i2}(j)} + E^{(i)}_{\pi_{i2}(j)j}(W_i u_{i+1}V_i)^\ast E^{(i)}_{j\pi_{i1}(j)}.
    \end{equation*}

    As $W_i, V_i \in R_i$, we observe that $u_{i+1} \in \{R_i, T_i\}''$, so $R_{i+1} = \{R_i, T_i\}''$. Note that $T_i$ is supported within $\sum_{j=M_i+1}^{M_i+m_i} E^{(i)}_{jj}$. By adding appropriate scalar multiples of $\sum_{j=M_i+1}^{M_i+m_i} E^{(i)}_{jj}$, we may assume the $T_i$ have the nonzero parts of their spectra disjoint from each other. Now, by the fact that $E^{(i)}_{jj} = \sum_{k=(j-1)K+1}^{jK} E^{(i+1)}_{kk}$, we see that the supports of $T_i$ for different $i$ are disjoint. Thus, by functional calculus, we see that all $T_i$, and therefore all $u_{i+1}$, are contained in the algebra generated by
    \begin{equation*}
        R_1, T = \sum_{i=1}^{n-1} T_i.
    \end{equation*}

    That is, $W^*(R,u_1,\ldots,u_n)$ is generated by $R,$ $u_1,$ and $T$. We note that $T$ is supported within $p = \sum_{i=1}^{n-1} \sum_{j=M_i+1}^{M_i+m_i} E^{(i)}_{jj} \in R$. And we observe that
    \begin{equation*}
        \tau(p) = \sum_{i=1}^{n-1} \frac{m_i}{K^i} \leq \sum_{i=1}^{n-1} (\sqrt{\frac{2}{K^i}} + \frac{2}{K^i}) < \ee,
    \end{equation*}
    completing the proof.
\end{proof}

\begin{thm}
    If $M$ is generated by a countable set of sequentially commuting Haar unitaries, then $M$ is singly generated.
\end{thm}

\begin{proof}
    Let $\{u_0, u_1, \cdots\}$ be an enumeration of a countable generating set of sequentially commuting Haar unitaries in which every unitary appears infinitely many times. By Lemma \ref{exist-of-hyperfinite-generator}, we may fix an irreducible $R \subset M$ s.t. $\{R, u_0\}''$ is generated by a unitary $v$ and a self-adjoint $T_0$ whose support is majorized by a projection $p_0 \in R$ with $\tau(p_0) < \frac{1}{2}$.
    
    Now, for each $k > 0$, let $n_k$ be such that $u_0 \sim_{n_k}^\cU u_k$. Let $\ee_k = \frac{1}{2^{k+1}}$. Choose $N = N_k > 0$ that satisfies the conclusions of Lemma \ref{generator-seq-commuting} with $\ee = \ee_k$ and $n = n_k$. Now, using Lemma \ref{lift-to-commuting-almost-haars} and the proof of Proposition \ref{equivalence of internal and external}, we see that there exists finite-dimensional unitaries $(x^{(k)}_i)_{i=1}^{n_k}$ and a Haar unitary $u'_k$ s.t. $[u_0, x^{(k)}_1] = [x^{(k)}_i, x^{(k)}_{i+1}] = [x^{(k)}_{n_k}, u'_k] = 0$ for all $i=1, \cdots, n_k-1$; $\|u'_k - u_k\|_2 < \ee_k$; and all atoms of $x^{(k)}_i$ have the same trace and the number of atoms are multiples of $N_k$, for all $i = 1, \cdots, n_k$. By Lemma \ref{generator-seq-commuting}, then, $\{R, u_0, x^{(k)}_1, \cdots, x^{(k)}_{n_k}, u'_k\}''$ is generated by $R$, $u_0$, and a self-adjoint $T_k$ whose support is majorized by a projection $p_k \in R$ with $\tau(p_k) < \ee_k = \frac{1}{2^{k+1}}$. As $\tau(p_0) + \sum_{k=1}^\infty \tau(p_k) < 1$, we may, after conjugating $T_k$ by unitaries in $R$, assume all $p_k$ (as well as $p_0$) are orthogonal to each other. For $k > 0$, by multiplying by appropriate scalars and adding appropriate scalar multiples, we may assume the nonzero parts of the spectra of $T_k$ are disjoint from each other and are all contained in a fixed compact interval disjoint from the spectrum of $T_0$. Then $T = T_0 + \sum_{k=1}^\infty T_k$ is a well-defined self-adjoint, and by functional calculus, we see that $\{R, u_0\}'' \subset \{v, T\}''$. Therefore, we also have $u'_k \in \{v, T\}''$ for all $k > 0$. Since $\{u_0, u_1, \cdots\}$ generates $M$, all generators appear infinitely many times, and $\|u'_k - u_k\|_2 < \ee_k \to 0$, we have $\{u_0, u'_1, u'_2, \cdots\}$ generates $M$, so $M$ is generated by a unitary $v$ and a self-adjoint $T$, whence it is singly generated (by $T + \ln(v)$, say).
\end{proof}

\begin{remark}\label{stability under amalgams}
    We note that a refinement of the above proof shows the following: Fix $\ee > 0$. Then, for $N > 0$ large enough, there exists a matrix subalgebra $\M_N(\C) \subset M$ s.t. $M$ is generated by the right shift unitary $v = \sum_{i=1}^{N-1} e_{i(i+1)} + e_{N1}$ in $\M_N(\C)$ and a self-adjoint $T$ whose support is majorized by a diagonal projection $p \in \M_N(\C)$ with trace less than $\ee$. This implies $\cG(M) = 0$ in the sense of \cite{shen2009type}. In particular, by Theorem 5.5 of \cite{shen2009type}, if $M$ is generated by a sequence of subfactors $M_k$, each one generated by a countable set of sequentially commuting Haar unitaries, and furthermore $M_k \cap M_{k+1}$ is diffuse for all $k$, then $M$ is singly generated.
\end{remark}

\bibliographystyle{amsalpha}
\bibliography{inneramen}

\end{document}